\documentclass{article}

\usepackage{amsmath}
\usepackage{latexsym}
\usepackage{verbatim}
\usepackage{amssymb}
\usepackage{graphicx}
\usepackage{stmaryrd}
\usepackage{tikz}
\newtheorem{theorem}{Theorem}

\newtheorem{corollary}[theorem]{Corollary}

\newtheorem{lemma}[theorem]{Lemma}

\newtheorem{proposition}[theorem]{Proposition}
\newtheorem{remark}[theorem]{Remark}

\newenvironment{proof}[1][Proof]{\noindent\textbf{#1.} }{\ \rule{0.5em}{0.5em}}
\begin{document}

\title{Wilf's Conjecture}
\markright{Wilf's Conjecture}
\author{Valerio De Angelis and Dominic Marcello}

\date{November 6, 2015}
\maketitle

\begin{abstract}
In a Note in this Monthly, Klazar raised the question of whether the alternating sum of the Stirling numbers of the second kind
$B^\pm(n)=\sum_{k=0}^n(-1)^kS(n,k)$ is ever zero for $n\neq 2$. In this article, we present an exposition of the history of this problem, and an economical account of a recent proof
that there is at most one $n\neq 2$ for which $B^\pm(n)=0$.
\end{abstract}

\newcommand{\bn}{B^{\pm}(n)}

\section{Introduction}

Entry A000587 in the Online Encyclopedia of
Integer Sequences \cite{OE} describes a relatively little known sequence. It was first recorded by
Ramanujan, who noted that it arises from the alternating version of the infinite series that he used to define the Bell numbers \cite[p.53]{BB}.  It was subsequently studied by many
authors (\cite{Be},\cite{UC},\cite{LP}, see \cite{DLO} for an extensive list of references). The terms of the sequence were sometimes called {\em complementary Bell numbers}, other times
{\em Rao Uppuluri-Carpenter numbers}, and yet other times they were not given any special name.

More recently, the sequence was studied in connection with product partitions of integers (\cite{HS},\cite{SV},\cite{Yang}). Klazar
proved results on the ordinary generating function \cite{Kla2}.

Sequence A000587 is also the object of study in Section 4 of Klazar's Monthly Note \cite{Kla}, where it is introduced as an example of sequences recording the surplus of objects
of even size over those of odd size. In that example, the objects are the partitions of $\{1,2,\ldots, n\}$ into $k$ disjoint subsets (called {\em blocks}).
The number of such objects is counted by the {\em Stirling numbers of the second kind} $S(n,k)$, whose sums over $k$
\[B(n) = \sum_{k=0}^n S(n,k)\]
are the {\em Bell numbers}, counting the total number of ways to partition $\{1,2,\ldots, n\}$ into disjoint subsets. Borrowing the notation from Klazar's article
(that will be used in this article),
sequence A000587 is
\[B^{\pm}(n)=\sum_{k=0}^n (-1)^k S(n,k).\]

So, $B^{\pm}(n)$
records the difference between the number of partitions of $\{1,2,\ldots, n\}$ with an even number of blocks and those with an odd number of blocks.
For example, there is one partition of $\{1,2\}$ with one block and one partition with two blocks, and so $B^{\pm}(2)=1-1=0$.
The first few terms of the sequence  are
\[
\begin{array}{c|ccccccccccc}
n &0 & 1 & 2&3&4&5&6&7&8&9&10\\
\hline
B^{\pm}(n) & 1 &-1& 0& 1& 1& -2& -9& -9& 50& 267& 413
\end{array}.
\]

 A question attributed to H. Wilf in \cite{Yang} asks whether $B^{\pm}(n)=0$ for only finitely many values of $n$.
Adamchik \cite{Adam} asked the more restrictive question of whether $n=2$ is the {\em only} value of $n$ for which $B^{\pm}(n)=0$.
 Klazar asked the same question in the Monthly note \cite{Kla} mentioned before.
 In a different context, this question is relevant to a result of Egorychev and Zima  \cite{EZ}, who showed that the sum $\sum_{k=1}^n k^a k!$
admits a closed form expression in terms of elementary functions only if $B^{\pm}(a+1)=0$. Note that  $\sum_{k=1}^n k  k!=(n+1)!-1$, in agreement with $B^{\pm}(2)=0$.

Progress was made by De Wannemacker, Laffey, and Osburn \cite{DLO}, who showed that $B^{\pm}(n)\neq 0$ if $n\neq 2$ or $2944838 \pmod{3\cdot 2^{20}}$.
The question (that has become known as {\em Wilf's conjecture})
has now been almost settled, using different methods, by Alexander \cite{Ale}, An \cite{An}, and Amdeberhan, De Angelis and Moll \cite{ADM}.
More precisely, it has been proved that there is at most one integer $n>2$ such that $B^{\pm}(n)=0$.

This article presents an economical account of the proof of the main result of \cite{ADM}, namely that $B^{\pm}(n)\neq 0$ for all $n>2$, with the possible exception
of one value. It also corrects some errors that were present in the original proof, and provides some additional numerical data
obtained using the Louisiana Optical Network Initiative high performing computer system.

For a positive integer $n$, denote by
$\nu_2(n)$ the exponent of the largest power of 2 that divides $n$. The integer $\nu_2(n)$ is called the {\em 2-adic valuation } of $n$.
Also define $\nu_2(0)=\infty$. Our strategy is to show that $\nu_2(B^{\pm}(n))$ is finite for all but at most one value of $n>2$.

The route to this result is summarized and illustrated by the  tree shown in Figure \ref{fig-tree1}.
The top three edges of the tree correspond to the residue class of
$n \pmod{3}$. The number by the side of the edge (if present) gives the
(constant) 2-adic valuation of
$B^\pm(n)$ for that residue class. For example
$\nu_2(B^\pm(3n+1))=0$. If there is no number next to the edge,
the 2-adic valuation is not constant for that residue class, so $n$ needs to
be split further. The split at each stage is conducted by replacing the
index $n$ of the sequence by $2n$ and $2n+1$. For example, the sequence
$\nu_2(B^\pm(12n+2))$ is not constant so it generates the two new
sequences
$\nu_2(B^\pm(24n+2))$ and
$\nu_2(B^\pm(24n+14))$.  Constant sequences include
$\nu_2(B^\pm(12n+8))=\nu_2(B^\pm(12n+5))=1$ and
$\nu_2(B^\pm(12n+11))=2$.

\begin{figure}[ht]
\begin{center}
\includegraphics[scale=0.55]{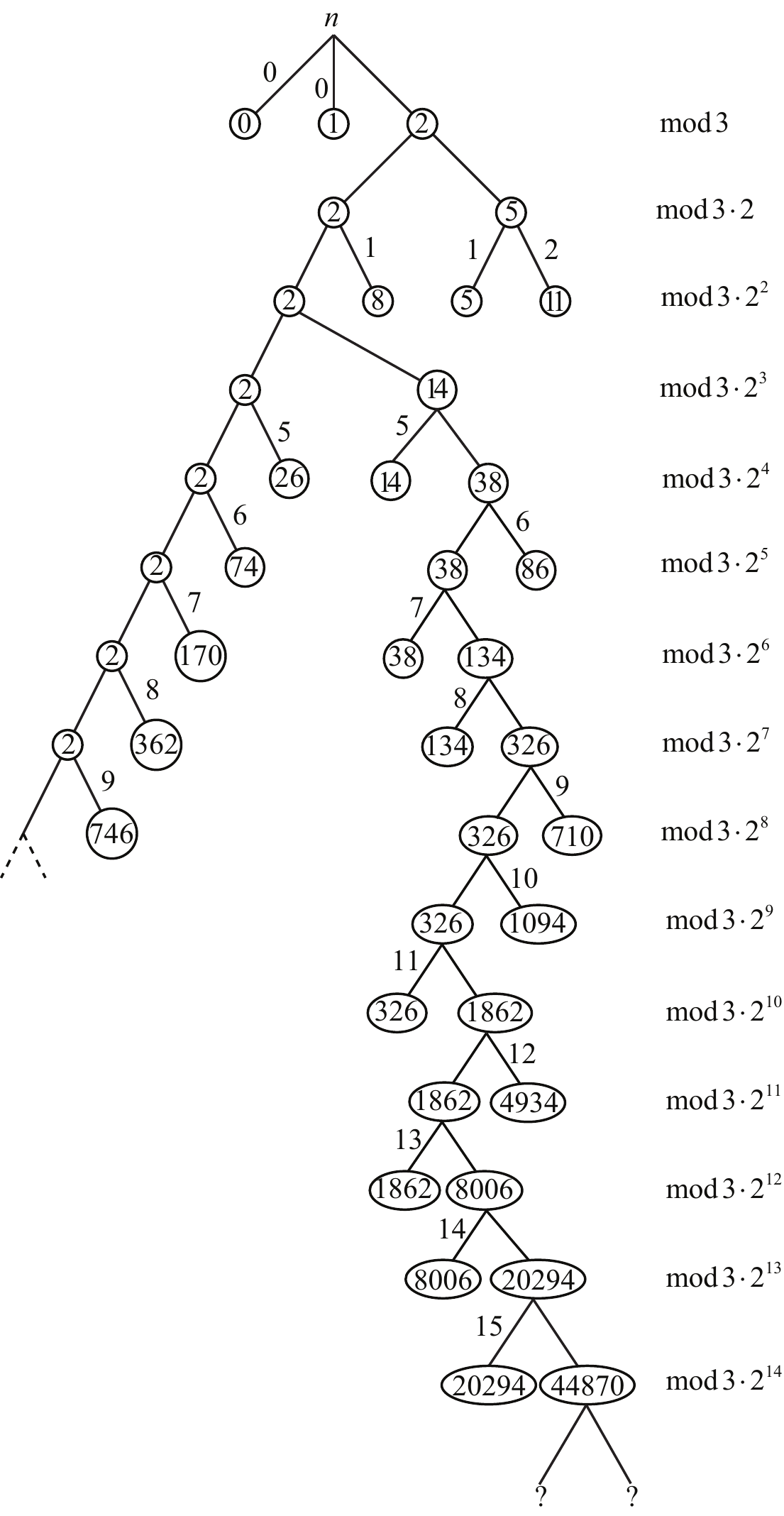}
\caption{The $2$-adic valuation of $B^{\pm}(n)$}
\label{fig-tree1}
\end{center}
\end{figure}

Theorem \ref{24n2} shows that the infinite branch on the left, coming
from the splitting of $24n+2$, has a well-determined structure, implying that $B^\pm(24n+2)\neq 0$  for all $n>0$. The
other infinite branch, corresponding to $24n+14$, exhibits an unpredictable pattern of left and right turns.
Theorem \ref{24n14} shows that if there are infinitely many right turns, then $B^\pm(n)\neq 0 $ for all $n\neq 2$.
But if there are only finitely many right turns, then there will be a unique integer $n^*>2944838$ such that $B^\pm(n^*)=0$, and the
binary digits of $n^*$ correspond to the left or right turns of the tree.

The main tool used throughout the paper  is the representation of $B^\pm(n)$ as the top left entry of the $n$-th  power of an infinite matrix $P$. Section 2 derives this representation, while Sections 3 and 4 develop suitable notation and
 describe symmetry properties of powers of $P$. The powers of form $P^{3\times 2^m}$ have a surprising structure when computed modulo a power of 2. Section 5 derives the results on this structure that are needed for the
 proofs of the two main theorems, that are presented in Sections 6 and 7.

Some computations in Section 5 depend on the computer algebra system Mathematica. As mentioned in that section, it is possible to derive more results on the first few powers of $P$ that would make
all computations accessible to a reader who does not wish to use a computer, and has a generous amount of time and enthusiasm to perform lengthy and routine algebra calculations. But this would make the paper longer by a few pages.
We have chosen not to follow that route, and rely instead on Mathematica to perform calculations that are not reasonably accessible by hand.

\section{Matrix representation for $B^{\pm}({\MakeLowercase n})$}
The recurrence relation for the Stirling numbers of the second kind,
\begin{equation}
S\left( n+1,k\right) =S\left( n,k-1\right) +kS\left( n,k\right),   \label{S2}
\end{equation}
can be proved by considering first the partitions of $\{1,\ldots, n+1\}$ for which the element $n+1$
forms a block by itself (giving the first term on the right), and then the other partitions.

Define monic polynomials $\lambda _{n}\left( x\right) $ (of degree
$n$ and with integer coefficients) recursively by%
\begin{eqnarray}
\lambda _{0}\left( x\right) &=&1,  \notag \\
\lambda _{n+1}\left( x\right) &=&x\lambda _{n}\left( x\right) -\lambda
_{n}\left( x+1\right) .  \label{lambda}
\end{eqnarray}%

As shown in the next lemma, $B^{\pm}(n)$ can be expressed as the constant term of this family of polynomials. The proof is a simple induction argument using (\ref{S2}).
\begin{lemma}
\label{b}
For each $n\geq 0$,
\begin{equation}
 B^{\pm}(n+j)=\sum\limits_{k=0}^{j} \left(
-1\right) ^{k}\lambda _{n}(k)S\left( j,k\right)  \ \ \mbox{for all} \ j\geq 0.  \label{1}
\end{equation}
In particular,
\[B^{\pm}(n)=\lambda_n(0).\]
\end{lemma}

The {\em descending factorial} $\left( x\right) _{r}$ is defined by
$\left(x\right) _{0}=1$, $\left(x\right) _{r}=x\left( x-1\right) \cdots \left( x-r+1\right) $, $r\geq 1$.
The following recurrence relations for $\left( x\right) _{r}$ are readily proved from the definition.

\begin{lemma}
\label{2} $\left( x\right) _{r}$ satisfies the recursions%
\[
x\left( x\right) _{r} =\left( x\right) _{r+1}+r\left( x\right) _{r}, \ \ \ \ \ \
\left( x+1\right) _{r} =\left( x\right) _{r}+r\left( x\right) _{r-1}.
\]
\end{lemma}

Using these recursions, $B^{\pm}(n)$ can be expressed as the top left entry of the $n$-th power
of an infinite matrix.

\begin{proposition}
\label{P}
Let $P$ be the infinite matrix indexed by the non-negative integers with entries $P(r,s), r,s\geq 0$,  defined by
\[
\begin{array}{lcl}
P(r+1,r)&=&1\\
P(r,r)&=&r-1\\
P(r,r+1)&=&-r-1\\
P(r,s)&=&0 \ \  \mbox{for } |r-s|>1,
\end{array}
\]
or
\begin{equation*}
P =\left(
\begin{array}{ccccccc}
-1 & -1 & 0 & 0 & 0 & 0 & \cdots\\
1 & 0 & -2 & 0 & 0 & 0 & \cdots\\
0 & 1 & 1 & -3 & 0 & 0 & \cdots\\
0 & 0 & 1 & 2 & -4 & 0 & \cdots\\
0 & 0 & 0 & 1 & 3 & -5 & \cdots\\
0 & 0 & 0 & 0 & 1 & 4 & \cdots \\
\vdots & \vdots & \vdots & \vdots & \vdots & \vdots & \ddots
\end{array}%
\right).
\end{equation*}
Then
\[B^{\pm}(n)=P^n(0,0).\]
\end{proposition}

\begin{proof}
Express the polynomials $\lambda_n(x)$ in terms of the descending factorial
\[\lambda_n(k)=\sum_{r=0}^n c_n(r) (k)_r\] where $c_n(r)$ are integers  with $c_0(0)=1$, $c_0(r)=0$ for $r>0$ and $c_n(r)=0$ if $r>n$.
Writing {\boldmath $c$}$_n$ for the vector $(c_n(r):r\geq 0)$, the recursive definition of $\lambda_n(x)$ and  Lemma \ref{2} yield the matrix equation
$\mbox{\boldmath $c$}_{n+1}=P\mbox{\boldmath $c$}_n.$
Note that powers of $P$ are well-defined, because each row or column has only finitely many non-zero entries.
So by iterating $c_{n+1}=Pc_n$ we obtain $\mbox{\boldmath $c$}_n=P^n\mbox{\boldmath $c$}_0,$
or $c_n(r)=P^n(r,0), r\geq 0.$
Since clearly $c_n(0)=\lambda_n(0)$, the result follows from Lemma \ref{b}.
\end{proof}

\section{$\MakeLowercase j$-diagonal matrices and block forms}

The study of $\bn$ is thus reduced to the study of powers of the infinite matrix $P$. This matrix has the property that $P(r,s)=0$ if
$|r-s|>1$. An infinite matrix $A$ having the property that $A(r,s)=0$ if $|r-s|>j$ will be called {\em $j$-diagonal}.
So, the matrix
$P$ is 1-diagonal.
A $j$-diagonal matrix has at most $2j+1$ non-zero entries on each row or column. Hence powers of the matrix are well-defined.

This section derives properties of $j$-diagonal matrices and their powers. It will be convenient to describe matrices in block form, and
some notation is now introduced.

Given intervals $I$ and $J$ over the non-negative integers, $A(I,J)$ denotes the submatrix of $A$ consisting of the entries $A(i,j)$ for
$i\in I$ and $j\in J$.  Let $n$ be a positive integer, and define the
 intervals  $I_k=\{kn, kn+1, \ldots ,(k+1)n-1\}$, $k\geq 0$.
 The {\em $n$-block form}
of $A$ is the matrix $A[n]$ whose $(r,s)$ entry
is the matrix $A(I_r,I_s)$.

Sometimes it is useful to describe a matrix in block form pictorially, as:

\[A=\left( \begin{array}{ccc}
\stackrel{i\times i}{\overbrace{X}} & Y & \ldots \\
Z & W & \ldots\\
\vdots & \vdots &
\end{array}\right),
\]
meaning that entries are $i\times i$ matrices.

This notation will be extended to rectangular blocks, so that
\[A=\left( \begin{array}{ccccc}
\stackrel{i\times i}{\overbrace{X_{00}}} & \stackrel{i\times j_1}{\overbrace{X_{01}}} & \cdots & \stackrel{i\times j_k}{\overbrace{X_{0k}}} &\cdots\\
X_{10} & X_{11} & \cdots & X_{1k} & \cdots \\
\vdots & \vdots &  &\vdots &
\end{array}\right),
\]
means that $X_{0k}$ is an $i\times j_k$ matrix.
When rectangular blocks are used, it will always be assumed that the size of the block at position $(r,s)$ is the size of the transpose of
the block at position $(s,r)$, while the diagonal blocks will always be square. With this convention, it is sufficient to specify the
size of the blocks in the top row, and
it is easily verified that the multiplication of two matrices written in the same block form can be done with the usual rule for matrix multiplication, where two individual blocks will
also be multiplied according to matrix multiplication.

A $j$-diagonal infinite matrix $A$ is completely described by the $2j+1$ vectors over the non-negative integers corresponding to the main diagonal
and the $j$ diagonals on each side of it. For $i\geq 1$, the vector $A(r+i,r):r\geq 0$ will be called the
 {\em $i$-th lower diagonal},  and the vector $A(r,r+i)$ will be called the $i$-th upper diagonal.

\begin{lemma}
If $A$ is a $j$-diagonal matrix and $B$ is a $k$-diagonal matrix, then $AB$ is a $j+k$-diagonal matrix.
\end{lemma}
\begin{proof}
Suppose $k\geq i+j$. Then $(AB)(r,r+k)=\sum_{s\geq 0}A(r,s)B(s,r+k)$. If $A(r,s)\neq 0$, then $|r-s|\leq i$, and so
$|r+k-s|\geq k-|r-s|>i+j-i=j$. Hence $B(s,r+k)=0$. Similarly, $(AB)(r+k,r)=0$.
\end{proof}

\begin{corollary} \label{jdiag}
For each $j\geq 1$, the matrix $P^j$ is $j$-diagonal.
\end{corollary}

The next result highlights an essential symmetry property of the matrix $P$. The proof is easily carried out by induction, by writing $P^{j+1}=P^jP$.

\begin{lemma}
For all $j\geq 1$, $r\geq 0$ and $s\geq 0$, the following relation holds:
\begin{equation}\label{sym}
(-1)^r r! P^j(r,s) = (-1)^s s! P^j(s,r).
\end{equation}
\end{lemma}

\section{2-adic valuation matrices and periodicity}

Let $s_2(n)$ be the number of 1's in the binary expansion of $n$.
The following classical formula by Legendre will often be used.
\begin{equation} \label{Leg}
\nu_2(n!) = n-s_2(n).
\end{equation}

An immediate consequence of  (\ref{sym}) is that the elements above the main diagonal of $P^j$ are highly divisible by 2, as made precise
by the following lemma.

\begin{lemma} \label{TR}
For each $r\geq 0$, $i\geq 1$, and $j\geq 1$,
\[\nu_2(P^j(r,r+i)\geq \nu_2(i!).\]
\end{lemma}
\begin{proof}
Using (\ref{sym}),
\begin{eqnarray*}
\nu_2\left(P^j(r,r+i)\right)&=& \nu_2\left(\frac{(r+i)!}{r!}\right)+\nu_2\left( P^j(r+i,r)\right)\\
&=&\nu_2\left(\binom{r+i}{r}\right)+\nu_2(i!) + \nu_2\left( P^j(r+i,r)\right)\\
&\geq & \nu_2(i!),
\end{eqnarray*}
since both $\binom{r+i}{r}$ and $P^j(r+i,r)$ are integers.
\end{proof}

A {\em valuation matrix} is any matrix over ${\mathbb Z}^+ \cup \{\infty\}$. If $A$ is a matrix over $\mathbb Z$,
$\nu_2(A)$ is the valuation matrix defined by $\nu_2(A)(r,s)=\nu_2(A(r,s))$.
If $k\in {\mathbb Z}^+\cup \{\infty \}$,
the notation $\nu_2(A)\geq k$ will mean that each entry of $A$ has 2-adic valuation at least $k$. For matrices $A$, $B$,
$\nu_2(A)\geq \nu_2(B)$ will mean that $\nu_2(A(r,s))\geq \nu_2(B(r,s))$ for all $r,s$.

A {\em $j$-diagonal valuation matrix} is a matrix that arises as
the 2-adic valuation of a $j$-diagonal matrix over the integers. So a $j$-diagonal valuation matrix $M$ is such that $M(r,s)=\infty$ whenever $|r-s|>j$.

For valuation matrices $M$, $N$, define the valuation matrix $M*N$ by
\[(M*N)(r,s)=\min\{M(r,t)+N(t,s):t\geq 0\}.\]
Then it follows from the definitions that for  matrices $A,B$ over the integers, $\nu_2(AB)\geq \nu_2(A)*\nu_2(B).$
Also define  $M+\mbox{\boldmath $1$}$ to be the matrix $M$ with each entry increased by 1.
\begin{lemma} \label{sqr}
Suppose $A$ is a matrix over $\mathbb Z$, $M$ is a valuation matrix, and \linebreak
$\nu_2\left(A-I\right)\geq M$. Then $\nu_2\left(A^2-I\right)\geq \min\{M+\mbox{\boldmath $1$},M*M\}$.
\end{lemma}
\noindent {\it Proof.}
Write $A=I+R$, so that $\nu_2(R)\geq M$. Then $A^2=I+2R+R^2$, and so
\begin{eqnarray*}\nu_2\left(A^2-I\right)& =&\nu_2\left(2R+R^2\right)  \geq \min\{\nu_2(2R),\nu_2(R^2)\}\\
&\geq & \min\{\nu_2(R)+\mbox{\boldmath $1$} ,\nu_2(R)*\nu_2(R)\} \geq  \min\{M+\mbox{\boldmath $1$},M*M\}.
\ \ \rule{0.55em}{0.55em}
\end{eqnarray*}
If all diagonals of a $j$-diagonal matrix are periodic, and $d$ is a common period for the diagonals, then
the matrix is determined by its top left corner of size $d+j$. The notation $A_k$ will denote the submatrix
$A(r,s):0\leq r,s \leq k-1$.

Suppose now that $B$ is a $(j+1)\times (j+1)$ matrix. Borrowing the notation for repeating decimals,
 $\overline{B}$ denotes the
$j$-diagonal infinite matrix obtained by extending
each diagonal of $B$ periodically. So, if $d$ is a period for all diagonals of $A$, then
$A=\overline{A_{j+d}}$. In particular, if all diagonals are constant, then $A=\overline{A_{j+1}}$.

This notation can also be used for matrices whose diagonals are periodic after an initial stretch.
So for example if $A$ is 2-diagonal, then
\[ A=\begin{pmatrix}
2&3&4&\infty\\
\cline{2-4}
2&1&1&1\\
0&1&1&2\\
\infty&0&1&1
\end{pmatrix}
\]
means that $A(0,r)$ and $A(r,0)$ are as shown by the topmost row and leftmost column, while for $r>0$,
 $A(r,r)=1$, $A(r,r+1)$ is alternating 1 and 2, $A(r,r+2)=A(r+1,r)=1$, and $A(r+2,r)=0$.

\section{The powers $P^{3\cdot 2^{\MakeLowercase m}}$}
Powers of $P$ of the form $3\cdot2^m$ play a special role. So define $d_m=3\cdot 2^m$. The main result of this section is
Proposition \ref{T8}, describing the top left $8\times 8$ corner of $P^{d_m}$ modulo $2^{m+3}$. The proof of that proposition will rely on an induction argument on $m$.
 However, in order to carry out the inductive step,
$m$ needs to be at least 4. This will require having a lower bound on the valuation of $P^{d_4}-I$ in order to start the induction.
Note that $P^{d_4}$ is a 48-diagonal infinite matrix, and so it has
up to 97 non-zero entries on each row or column. It is possible to derive the required lower bound for  $\nu_2(P^{d_4}-I)$ with calculations that can be performed by hand,
by first deriving additional properties of small powers of $P$.  We have chosen instead to rely on Mathematica
to perform computations that provide the required lower bound for $\nu_2(P^{d_4})$.

 The first step is the derivation of
explicit formulas for the non-zero diagonals of $P^{d_0}=P^3$.
A simple computation from the definition of $P$ gives the main diagonal and the three lower diagonals, and Lemma \ref{sym} gives the three upper diagonals.

\begin{lemma}
\label{P3L}
For all $r\geq 0$, we have:
\[
\begin{array}{ll}
P^3(r+3,r)=1 & P^3(r,r+1)=(r+1)(2+6r-3r^2)\\
P^3(r+2,r)=3r&P^3(r,r+2)=3r(r+1)(r+2)\\
P^3(r+1,r)=3r^2-6r-2& P^3(r,r+3)=-(r+3)(r+2)(r+1).\\
P^3(r,r)=r^3-9r^2+6r+1
\end{array}
\]
\end{lemma}

We say that a matrix $A$ is {\em strictly lower triangular} if $A(r,r+i)=0$ for all $r\geq 0$, $i\geq 0$. A strictly lower triangular matrix modulo an integer is similarly defined.
The following lemma follows immediately from the definitions.
\begin{lemma} \label{A4}
Suppose $A$ is strictly lower triangular. Then
\begin{enumerate}
\item[(i)]
$A^2$ is strictly lower triangular,
\item[(ii)]
The first lower diagonal of $A^2$ is zero.
\item[(iii)]
The first three lower diagonals of $A^4$ are zero.
\end{enumerate}
\end{lemma}

\begin{lemma} \label{m7}
\begin{enumerate}
\item[(a)]
The 2-block matrix $P^3[2]-I$ is strictly lower triangular $\pmod{2}$.
\item[(b)]
All upper diagonals, the main diagonal, and the three lower diagonals of $P^{d_4}[8]-I$ are even.
\item[(c)]
If $m\geq 3$, $0\leq r \leq 7$ and $s\geq d_{m-1}+8$, then $\nu_2(P^{d_m})\geq m+7$.
\end{enumerate}
\end{lemma}
\begin{proof}
\begin{enumerate}
\item[(a)]
Note that $P^3[2]$ is a 2-diagonal matrix.
The $(r,r)$ entry of the 2-block form of $P^3-I$ is
$$
\begin{pmatrix}
P^3(2r,2r)& P^3(2r,2r+1)\\
P^3(2r+1,2r) & P^3(2r+1,2r+1).
\end{pmatrix}
$$
It is easily verified from the formulas of Lemma \ref{P3L}  that all entries are even.
In a similar way, it is easy to check that the $(r,r+1)$ and $(r,r+2)$ entries of $P^3[2]$ are even.
\item[(b)]
From part (a), $A=P^3[2]-I$ is strictly lower triangular $\pmod{2}$. Hence from Lemma \ref{A4},  $P^6[2]-I\equiv A^2 \pmod{2}$ is strictly lower triangular $\pmod{2}$
and its first lower diagonal is even.
It follows
that $P^6[4]-I$ is strictly lower triangular $\pmod{2}$. Repeating the same argument, $B=P^{12}[8]-I$ is strict lower triangular. Hence from Lemma \ref{A4},
$P^{48}[8]-I\equiv B^4 \pmod{2}$ is strictly lower triangular $\pmod{2}$, and its first three lower diagonals are even.
\item[(c)]
If $s\geq d_{m-1}+8$ and $0\leq r \leq 7$,  then $s-r\geq d_{m-1}+1$. The binary expansion of $d_{m-1}+1=2^m+2^{m-1}+1$ has  three 1's, for all $m\geq 4$.
Hence by Lemma \ref{TR} and Legendre's formula (\ref{Leg}), $\nu_2(P^{d_m})(r,s)\geq \nu_2((d_{m-1}+1)!)=3\cdot 2^{m-1} +1-3\geq m+7$ for $m\geq 3$.
\end{enumerate}
\end{proof}

The following lemma gives a precise description of  when it is possible to replace the infinite matrix $P^i$ with the finite matrix $(P_n)^i$. The proof is done by induction on $i$, by writing
$(P_n)^{i+1}=(P_n)^iP_n$.
\begin{lemma}  \label{Pn}
For each $n\geq 1$ and $i\geq 1$,
\[\left(P_n\right)^i(r,s) = P^i(r,s) \ \ \mbox{for} \ 0\leq r,s \leq n-1, \ \ r+s+i\leq 2n-1.\]
\end{lemma}

The next proposition provides the required lower bound on $\nu_2(P^{d_4}-I)$ mentioned before.
When writing a matrix in block form, an entry consisting of an integer is understood to mean a matrix (of the appropriate size)
all of whose entries are equal to that integer.

\begin{proposition} \label{m4}
\begin{equation} \label{P4lb}
\nu_2\left(P^{d_4}-I\right)\geq
\left(\begin{array}{ccccc}
 \overbrace{3}^{8\times 8}& \overbrace{7}^{8\times 24}&  \overbrace{11}^{8\times 24} &  \overbrace{11}^{8\times 24} & \overbrace{\infty}^{8\times 24}\\
\cline{2-5}
1& 1 & 1& 1& 1\\
0&0&1&1&1\\
0&0&0&1&1\\
\infty &0&0&0&1
\end{array}\right).
\end{equation}
\end{proposition}

\begin{proof}
Lemma \ref{m7} (c) shows that a lower bound for the third and fourth entry of the top row of $P^{d_4}$ is $4+7=11$.

From Lemma \ref{m7} (b), all entries on the main and upper diagonals of $P^{d_4}-I$, as well as all entries in the top left block of size $32\times 8$,
 are even. So it remains to show that
lower bounds for the 2-adic valuation of the first two entries of the top row are given by $3$ and $7$, respectively. According to Lemma \ref{Pn}, all $(r,s)$ entries
in the $48$-th  power of $P_n$ are equal to $P^{48}(r,s)$ provided $r+s+48\leq 2n-1$. The relevant bounds for $r$ and $s$ are $r\leq 7$ and $s\leq 31$,
hence $7+31+48\leq 2n-1$ gives $n\geq 44$. A computation with Mathematica of the $48$-th power of the top left corner of $P$ of size $44\times 44$ $\pmod{2^7}$
completes the proof.
\end{proof}

\begin{proposition} \label{Pdm}
If $m\geq 4$, then
\begin{equation} \label{Prop}
\nu_2\left(P^{d_m}-I\right)\geq
 \begin{pmatrix}
 \overbrace{m-1}^{8\times 8}& \overbrace{m+3}^{8\times d_{m-1}}&  \overbrace{m+7}^{8\times d_{m-1}} &  \overbrace{m+7}^{8\times d_{m-1}} & \overbrace{\infty}^{8\times d_{m-1}}\\
 \cline{2-5}
 1&1&1&1&1\\
 0&0&1&1&1\\
 0&0&0&1&1\\
 \infty&0&0&0&1
\end{pmatrix} .
\end{equation}
\end{proposition}
\begin{proof}
Let $M$ be the right side of (\ref{Prop}). As in the case $m=4$, Lemma \ref{m7} (c)  establishes the
lower bound for the third and fourth entries shown in the top row of $M$.

The rest of the proof is by induction on $m$, with the case $m=4$ given by Proposition \ref{m4}. Suppose that the result holds for some $m\geq 4$. A routine computation shows that a
lower bound for $\min\{M+\mbox{\boldmath $1$},M*M\}$ is given by
\[
 \begin{pmatrix}
 \overbrace{m}^{8\times 8}& \overbrace{m+4}^{8\times d_{m-1}}&  \overbrace{m+4}^{8\times d_{m-1}} &  \overbrace{m+4}^{8\times d_{m-1}} & \overbrace{m+4}^{8\times d_{m-1}}
& \overbrace{m+5}^{8\times d_{m-1}}&\overbrace{m+5}^{8\times d_{m-1}}&\overbrace{\infty}^{8\times d_{m-1}}\\
 \cline{2-8}
 1&1&1&1&1&1&1&1\\
  1&1&1&1&1&1&1&1\\
 0&0&1&1&1&1&1&1\\
 0&0&0&1&1&1&1&1\\
 0&0&0&0&1&1&1&1\\
 0&0&0&0&0&1&1&1\\
 \infty&0&0&0&0&0&1&1
\end{pmatrix}.
\]
Rearrange the block sizes so that on the top row, the first entry is $8\times8$, and the
other entries are $8\times 2d_{m-1}=8\times d_m$, and replace each block with the minimum of its entries. Then an application of Lemma \ref{sqr} gives the result. \end{proof}

By Proposition \ref{m4},  $\left(P^{d_4}-I\right)_8$ is divisible by 8. So we can define
$Q$ to be the matrix with entries in $\{0,1,\ldots, 15\}$ such that $\left(P^{d_4}\right)_8\equiv I+8Q \pmod{2^7}$. The same computations with Mathematica
used for the proof of Proposition \ref{m4} can be used to find the entries of $Q$, that can then be squared to complete the proof of the following lemma.

\begin{lemma}\ \label{Q} With Q defined as above, we have:
\begin{itemize}
\item[(i)]
$Q=\begin{pmatrix}
2& 4& 12& 0&8&8&0&0\\
12&2&0&12&8&0&0&0\\
6&8&10&8&0&8&8&8\\
0&2&8&6&0&0&8&0\\
13&5&0&12&10&12&12&8\\
3&6&14&8&4&2&0&12\\
9&11&13&11&2&0&2&0\\
2&15&11&2&12&6&0&6
\end{pmatrix}.$
\item[(ii)]
$Q^2\equiv 0 \pmod{4}$.
\end{itemize}
\end{lemma}

\begin{proposition} \ \label{T8}
For each $m\geq 4$,
$\left(P^{d_m}\right)_8\equiv I +2^{m-1} Q \pmod{2^{m+3}}$.
\end{proposition}

\begin{proof}
We write $O(2^k)$ to denote any number (or matrix) that is zero modulo $2^k$.
Note that
\[\left(P^{d_{m+1}}\right)_8=\left(\left(P^{d_m}\right)_8\right)^2 +\sum_{t=8}^{d_m} P^{d_m}(r,t)P^{d_m}(t,s).\]
By induction, and using Lemma \ref{Q},
\begin{eqnarray*}
\left(\left(P^{d_m}\right)_8\right)^2&=&\left(I+2^{m-1}Q+O\left(2^{m+3}\right)\right)^2=I+2^m Q+2^{2m-2}Q^2+O\left(2^{m+4}\right)\\
&=&I+2^mQ +O\left(2^{2m}\right) + O\left(2^{m+4}\right)\equiv I+2^mQ \pmod{2^{m+4}}.
\end{eqnarray*}
Also, using Proposition \ref{Pdm},
\begin{eqnarray*}
\sum_{t=8}^{d_m} P^{d_m}(r,t)P^{d_m}(t,s)&=&\sum_{t=8}^{7+d_{m-1}} P^{d_m}(r,t)P^{d_m}(t,s)+
\sum_{t=8+d_{m-1}}^{d_m}  P^{d_m}(r,t)P^{d_m}(t,s)\\
&=&\sum_{t=8}^{7+d_{m-1}} O\left(2^{m+3}\right)O(2)+\sum_{t=8+d_{m-1}}^{d_m} O\left(2^{m+4}\right)\\
&\equiv& 0 \pmod{2^{m+4}}.
\end{eqnarray*}
So the result follows.
\end{proof}

\begin{corollary} \label{ndm}
For each positive integer $n$ and all $m\geq 4$,
\[\left(P^{nd_m}\right)_8\equiv I + n2^{m-1}Q \pmod{2^{m+3}}\]
and
\[P^{nd_m}(r,s)\equiv 0 \pmod{2^{m+3}} \ \ \mbox{for all  } 0\leq r \leq 7, \ \ s\geq 8\]
\end{corollary}

\begin{proof}
It follows from Proposition \ref{Pdm} and Proposition \ref{T8} that there are matrices $X$, $Y$ such that
\[P^{d_m}\equiv I +
\begin{pmatrix}   \overbrace{2^{m-1}Q}^{8\times 8} &  \overbrace{0}^{8\times \infty}\\
X & Y
\end{pmatrix} \pmod{2^{m+3}}.
\]
If $A$ is the matrix on the right side of the above congruence, then
$A^k$ has the block form
\[A^k\equiv \begin{pmatrix}   \overbrace{0}^{8\times 8} &  \overbrace{0}^{8\times \infty}\\
X_k & Y_k
\end{pmatrix} \pmod{2^{m+3}}\]
for all   $k>1$.
Since
\[P^{nd_m} \equiv \left(I + A\right)^n =I+nA+\sum_{k=2}^n\binom{n}{k} A^k \pmod{2^{m+3}},\]
the result follows.
\end{proof}

\section{The 2-adic valuation of $B^{\pm}(24{\MakeLowercase n}+2)$}
In this section the information on $P^{d_m} \pmod{2^{m+3}}$ found in the previous section is used to derive results on the complementary Bell numbers. We will use the notation  $v_j=(Q(P^j)_8)(0,0)$, $j\geq 1$.

Since $B^{\pm}(j)=P^j(0,0)$, and $d_m\equiv 0 \pmod{48}$ for $m\geq 4$,  Corollary \ref{ndm} provides the  2-adic valuation of
$B^{\pm}(n)-1$ when $n\equiv 0 \pmod{48}$.  In order to
obtain results on other residue classes modulo 48, the following lemma will be needed.

\begin{lemma} \label{Vj} For $j\geq 0$,
\begin{eqnarray*}v_j&\equiv&11 B^{\pm}(j+1) +14B^{\pm}(j+2)+9 B^{\pm}(j+3)\\
&+&6B^{\pm}(j+4) + B^{\pm}(j+5) \pmod{16}\end{eqnarray*}
\end{lemma}
\begin{proof}
From Lemma \ref{Q}, $Q(0,r)\equiv 0 \pmod{16}$ for $6\leq r \leq 7$. Hence
\[v_j\equiv\sum_{r=0}^5 Q(0,r)P^j(r,0) \pmod{16}.\]
The proof  relies on expressing
$P^j(0,r)$, for $0\leq r \leq 5$, as a linear combination of $B^{\pm}(j+i)$, $0\leq i \leq r$.
The computation
\[B^{\pm}(j+1)=P^{j+1}(0,0)=P^j(0,0)P(0,0)+P^j(0,1)P(1,0)=-B^{\pm}(j)+P^j(0,1)\]
finds such linear combination for $r=1$, that is
\[P^j(0,1)=B^{\pm}(j)+B^{\pm}(j+1).\]
Then, using this equation with $j$ replaced by $j+1$, and expanding the product $P^{j+1}=P^jP$,
the linear combination for $r=2$ is obtained, that is
\[P^j(0,2)=B^{\pm}(j)+B^{\pm}(j+1)+B^{\pm}(j+2).\]
Continuing in this way, the linear combinations for the
 remaining values of $r$ are obtained:
 \begin{align*}
 P^j(0,3)&=B^{\pm}(j)+2B^{\pm}(j+1)+B^{\pm}(j+3)\\
 P^j(0,4)&=B^{\pm}(j)+ 5 B^{\pm}(j+2) - 2 B^{\pm}(j+3) +
 B^{\pm}(j+4)\\
 P^j(0,5)&=B^{\pm}(j)+ 9 B^{\pm}(j+1)- 15 B^{\pm}(j+2)\\
 & +15 B^{\pm}(j+3)- 5 B^{\pm}(j+4)+ B^{\pm}(j+5).
 \end{align*}

 Using Lemma \ref{sym} and some computations,
 \begin{eqnarray*}5! \sum_{r=0}^5 Q(0,r)  P^j(r,0) &=&5! \sum_{r=0}^5 Q(0,r)(-1)^r P^j(0,r)/r!\\
 &= & 8 [64 B^{\pm}(j)+ 21 B^{\pm}(j+1)+ 130 B^{\pm}(j+2)\\
 &-& 25 B^{\pm}(j+3) + 10 B^{\pm}(j+4) - B^{\pm}(j+5)].
 \end{eqnarray*}
 Hence
 \begin{eqnarray*}
 15 \sum_{r=0}^5 Q(0,r)  P^j(r,0)
 &=&  64 B^{\pm}(j)+ 21 B^{\pm}(j+1)+ 130 B^{\pm}(j+2)
 - 25 B^{\pm}(j+3) \\
 &+& 10 B^{\pm}(j+4) - B^{\pm}(j+5).
 \end{eqnarray*}
Considering the last equation modulo 16, the result follows.
\end{proof}

\begin{lemma} \label{qdm4}
 Let $m\geq 0$, $j\geq 0$, and $q\geq 1$. Then
\[B^{\pm}(qd_{m+4}+j)\equiv B^{\pm}(j)+2^{m+3}qv_j \pmod{2^{m+7}}.\]
\end{lemma}

\begin{proof}
By Proposition \ref{Pdm}, $P^{qd_{m+4}}(0,s)\equiv 0 \pmod{2^{m+7}}$ for all $s\geq 8$. It follows that
\begin{eqnarray*}
(P^{qd_{m+4}+j})_8&=& (P^{qd_{m+4}}P^j)_8\equiv  P_8^{qd_{m+4}} (P^j)_8 \\
&\equiv& \left(I+2^{m+3}qQ\right)(P^j)_8\equiv (P^j)_8+2^{m+3}qQ(P^j)_8 \pmod{2^{m+7}}.
\end{eqnarray*}
Now take the $(0,0)$ entry to get the result.
\end{proof}

\begin{corollary} \label{nu} For $j\geq 0$,
\[\nu_2\left(B^{\pm}(j)\right)=\left\{
\begin{array}{cc} 0 & \mbox{ if } j\equiv  0 \mbox{ or } 1\pmod{3}\\
1 & \mbox{ if } j\equiv 5 \mbox{ or } 8  \pmod{12}\\
2 & \mbox{ if } j\equiv 11  \pmod{12}
\end{array}\right.
\]
\end{corollary}

\begin{proof}
The previous lemma with $m=0$ and $q=1$ gives
\[B^{\pm}(48+j)\equiv B^{\pm}(j) \pmod{8}.
\]
Now the result follows by computing $\nu_2\left(B^{\pm}(j)\right)$ for $0\leq j \leq 47$.
\end{proof}

The last corollary provides the value of $\nu_2\left(B^{\pm}(j)\right)$ except for the case
$j\equiv 2$ $\pmod{12}$.
The next theorem provides a complete answer for half of the remaining cases, namely the case $j\equiv 2 $ $\pmod{24}$.
\begin{theorem} \label{24n2}
For each $n\geq 0$,
\[\nu_2\left(B^{\pm}(24n+2)\right)=\nu_2(n)+5.\]
\end{theorem}
\begin{proof}
If $n=0$, then both sides are $\infty$. Suppose now that $n\geq 1$. If $n$ is even, write $n=2^{m+1}q$, with
$m\geq 0$ and $q$ odd.
Then, using Lemma \ref{qdm4} and the values $B^{\pm}(2)=0$, $v_2=-56\equiv 8 \pmod{16}$,
\[B^{\pm}(24n+2)=B^{\pm}(qd_{m+4}+2)
\equiv B^{\pm}(2)+2^{m+6} q
\equiv  2^{m+6} \pmod{2^{m+7}}. \]
Hence $\nu_2(B^{\pm}(24n+2))=m+6 =\nu_2(n)+5$.

If $n$ is odd, write $n=2q+1$. Then Lemma \ref{qdm4} with $m=0$ gives
\[B^{\pm}(24n+2)= B^{\pm}(48q+26)\equiv B^{\pm}(26)+8q \pmod{2^7}.
\]
Using the computed valuations $\nu_2(B^{\pm}(26))=5$, $\nu_2(v_{26})=3$, it follows that
\[\nu_2(B^{\pm}(24n+2))=5. \]
This completes the proof.
\end{proof}

\section{The case $B^{\pm}(24 {\MakeLowercase n}+14)$}
It remains to consider the case $n\equiv 14 \pmod{24}$. Using Lemma \ref{qdm4} and the computed value
$\nu_2(v_{14})=3$, and proceeding as before, it is easy to see that if $n$ is even, then
$\nu_2(B^{\pm}(24n+14))=5$. But the case $n$ odd, say $n=2k+1$, splits into two cases, giving $\nu_2(B^{\pm}(48k+38))=6$ if
$k$ is odd, and $\nu_2(B^{\pm}(48k+38))\geq 7$ if $k$ is even. In fact, this process will continue indefinitely, as will soon be made precise.
Some values of $B^{\pm}(n)$ $\pmod{16}$ will be needed.

\begin{lemma}
If $j\equiv 38 \pmod{48}$, then
\[
\begin{array}{cc}
\left. \begin{array}{ccc}
B^{\pm}(j+1)&\equiv& 5  \\
B^{\pm}(j+2)&\equiv& 5  \\
B^{\pm}(j+3)&\equiv& 14  \\
B^{\pm}(j+4)&\equiv& 3  \\
B^{\pm}(j+5)&\equiv& 11
\end{array}\right\} .
& \pmod{16}
\end{array}
\]
\end{lemma}

\begin{proof}
Using the 2-adic valuations given by Corollary \ref{nu}, $B^{\pm}(j+3)$ is even, while $B^{\pm}(j+1)$ and $B^{\pm}(j+5)$ are odd. It follows
from Lemma \ref{Vj} that $v_j$ is even. Hence the congruence of Lemma \ref{qdm4} (with $m=0$ and $q=1$) gives
\[B^{\pm}(48+j)\equiv B^{\pm}(j) \pmod{16}.\]
Now the result follows by computing $B^{\pm}(38+i)\pmod{16}$ for $1\leq i \leq 5$.
\end{proof}

From this lemma and Lemma \ref{Vj} we immediately get the following consequence:
\begin{corollary} \label{38}
If $j\equiv 38 \pmod{48}$, then $\nu_2(v_j)=3$.
\end{corollary}

Define the sequence $y_m$ inductively by $y_0=1$,
and
\[y_{m+1} = \left\{ \begin{array}{lr} y_m & \mbox{if} \ \nu_2\left(B^\pm (24y_m+14)\right) > m+5 \\
2^m + y_m & \mbox{if} \ \nu_2\left(B^\pm(24y_m+14)\right) \leq m+5 \end{array}\right. .\]
Then define
\[x_m=24y_m + 14.\]
Note that $y_m$ is non-decreasing, and bounded above by $2^m-1$ for $m\geq 1$.

The following table gives the values of $y_m$ and $x_m$ for $m\leq 18$. This material is based upon work supported by the Louisiana Optical Network Institute (LONI).

\begin{center}
\[\begin{array}{c|ccccccccccc}
m& 0 & 1 & 2 & 3 & 4 & 5 & 6 & 7 & 8 & 9 & 10\\
\hline
y_m & 1 & 1 & 1 & 5 & 13 & 13 & 13 & 77 & 77 & 333 & 845\\
x_m & 38 & 38 & 38 & 134 & 326 & 326 & 326 & 1862 & 1862 & 8006 & 20294
\end{array}
\]
\[
\begin{array}{c|cccccccc}
m& 11 & 12 & 13 & 14 & 15 & 16 & 17 & 18 \\
\hline
y_m & 1869 & 3917 & 8013 & 801 & 24397 & 57165 & 122701 & 122701\\
x_m & 44870 & 94022 & 192326 & 192326 & 585542 & 1371974 & 2944838 & 2944838
\end{array}
\]
\label{LONI}
\end{center}

\begin{remark}
The sequence $B^\pm(x_m)$ seems to grow super-exponentially, as shown in the table below.
\[\begin{array}{c|ccccccccc}
m & 0 & 1 & 2 & 3 & 4 & 5 & 6 & 7 & 8 \\
\hline
\log_{10} |B^\pm(x_m)| & 27.3 &  27.3 & 27.3 & 153.8 & 475.8 &
475.8 & 475.8 & 3908.4 & 3908.4
\end{array}
\]
It has been shown by several authors that $|B^\pm(n)|$ is not bounded by any power of $n$.
A simple proof is given in \cite{Kla}. More precise asymptotic estimates are given in \cite{SV} and \cite{Yang}.
\end{remark}

\begin{lemma} \label{xm}
\[B^{\pm}(24\cdot 2^m n +x_m)\equiv B^{\pm}(x_m) + 2^{m+5}n \pmod{2^{m+6}}.\]
\end{lemma}

\begin{proof}
Suppose first that $m\geq 1$, and write $k=m-1$. Note that $x_m\equiv 38 \pmod{48}$, and so by Corollary \ref{38} $\nu_2(v_{x_m})=3$. Then Lemma \ref{qdm4} gives
\begin{eqnarray*}
B^{\pm}(24\cdot 2^m n +x_m) &=&B^{\pm}(d_{k+4}+x_m)\\
&\equiv& B^{\pm}(x_m)+2^{k+6}n \pmod{2^{k+7}}\\
&=& B^{\pm}(x_m)+2^{m+5}n \pmod{2^{m+6}}.
\end{eqnarray*}
If $m=0$, then the case $n=2k$ gives
\begin{eqnarray*}
B^{\pm}(24n+38)&=&B^{\pm}(48k+38)\equiv B^{\pm}(38)+2^6k \pmod{2^7}\\
&\equiv& B^{\pm}(38)+2^5n \pmod{2^6},
\end{eqnarray*}
 while the
case $n=2k+1$ gives
\[B^{\pm}(24n+38)=B^{\pm}(48k+62)\equiv B^{\pm}(62)+2^6k \pmod{2^7}.\]
Since $\nu_2(B^{\pm}(62))=5$ and $\nu_2(B^{\pm}(38))=7$, $B^{\pm}(62)+2^6k\equiv B^{\pm}(38)+2^5n \pmod{2^6}$,
and the result follows.
\end{proof}

\begin{lemma}
$\nu_2(B^{\pm}(x_m))\geq m+5$ for all $m\geq 0$.
\end{lemma}
\begin{proof}
If $m=0$, then $\nu_2(B^{\pm}(x_0))=\nu_2(B^{\pm}(38))=7$. Assume now that $\nu_2(B^{\pm}(x_m))\geq m+5$. If $\nu_2(B^{\pm}(x_m))> m+5$, then by definition
$x_{m+1}=x_m$, and so $\nu_2(B^{\pm}(x_{m+1}))=\nu_2(B^{\pm}(x_{m}))\geq m+6 =(m+1)+5$. If $\nu_2(B^{\pm}(x_m))= m+5$, then
write $B^{\pm}(x_m)=2^{m+5}q$ with $q$ odd. Then $x_{m+1}=24\cdot 2^m +x_m$, and using Lemma \ref{xm},
\[
B^{\pm}(x_{m+1})\equiv B^{\pm}(x_m)+2^{m+5}
=2^{m+5}q+2^{m+5} \equiv 0 \pmod{2^{m+6}}.
\]
So $\nu_2(B^{\pm}(x_{m+1}))\geq (m+1)+5$.
\end{proof}

Recall that $y_m$ is bounded above by $2^m-1$. Let the binary expansion of $y_m$ (for $m\geq 1$) be
\[y_m=\sum_{i=0}^{m-1} s_{m,i}2^i.\]
Since the binary expansion of $y_{m+1}$ is the same as that of $y_m$ with possibly an extra leading 1, the limit
\[s_i=\lim_{m\rightarrow \infty} s_{m,i}\]
exists for all $i$ (equivalently, the 2-adic limit of $y_m$ exists) . Define the infinite binary sequence (or 2-adic integer)
\[s=(s_0,s_1,s_2,\ldots)=(1,0,1,1,0,0,1,0,1,1\ldots).\]

\begin{theorem} \label{24n14}
Let $n$ be a positive integer with binary expansion $n=\sum_kb_k2^k$, and let $m=m(n)$ be the first index for which $b_k \neq s_k$ .
 If no such index exists, let $m=\infty$. Then
\begin{equation} \label{m5}
\nu_2(B^{\pm}(24n+14)=m+5.
\end{equation}
\end{theorem}
\begin{proof}
Suppose first that $m=\infty$. Then the sequence $s$ has only finitely many non-zero entries. If $i$ is the largest index such that $s_i=1$, then $x_m=x_i$ for all $m\geq i$.
This implies that $B^{\pm}(x_i)$ is divisible by arbitrarily large powers of $2$. Hence $B^{\pm}(x_i)=0$, and (\ref{m5}) holds.

If $m<\infty$, then  $24n+14=24\cdot 2^m p+x_m$, where $p=b_m+2b_{m+1}+2^2b_{m+2}+\cdots$. Then Lemma \ref{xm} gives
\[B^{\pm}(24n+14)=B^{\pm}(24\cdot 2^m p +x_m)\equiv B^{\pm}(x_m)+2^{m+5}p \pmod{2^{m+6}}.\]
If $\nu_2(B^{\pm}(x_m))>m+5$, then $y_{m+1}=y_m$, and so $s_m=0$. Hence $b_m=1$ and $p$ is odd. On the other hand, if
$\nu_2(B^{\pm}(x_m))=m+5$, then $y_{m+1}=2^m+y_m$, so $s_m=1$ and $b_m=0$, implying that $p$ is even. In both cases the last congruence shows that
$\nu_2(B^{\pm}(24n+14))=m+5$.
\end{proof}

The last theorem shows that if the sequence $s$ contains infinitely many 1's, then $m$ is always finite, and hence the 2-adic valuation of $B^{\pm}(n)$ is finite for $n>2$.
On the other hand, if the sequence $s$ terminates and $s_i$ is its last non-zero entry, then $x_i$ is the only integer greater than 2 such that $B^{\pm}(x_i)=0$.
The question of whether or not the sequence $s$ terminates has not been settled, and so the existence of an exceptional integer $n>2$ such that $B^{\pm}(n)=0$ is still
an open question.

This question seems difficult to settle because there is no discernible pattern in the even-odd splitting of the subsequences of $24n+14$ that provide an increasing
2-adic valuation for $B^{\pm}(24n+14)$. However, some natural questions remain open for further exploration, especially with the
help of a computer algebra system, and their resolution could shed some light on the sequence $s$.

\begin{enumerate}
\item This paper derived results on the structure of the top left corner of $P^{3\cdot 2^m} = P^{2^{m+1}} P^{2^m}$ modulo a suitable power of 2. Can some results be derived about the top left corner of $P^{2^m}$? Ideally, enough information on the structure of the
individual powers of form $2^m$ could result in a better understanding of a generic power $n$, simply by considering the binary
expansion of $n$.
\item
It is easy to generalize the construction of Lemma \ref{Vj} and express $P^j(0,r)$ as a linear combination of $B^{\pm}(j+i)$, $0\leq i \leq r$,
for any $r$ (see Lemma 10.2 of \cite{ADM}), with coefficients defined by a recurrence relation. Is it possible to extract some more
information on the 2-adic valuation of the top row of $P^j$ from this recurrence?
\item
If $e_m=\nu_2\left(B^{\pm}(x_m)\right)$, what can be said about the odd part $B^{\pm}(x_m)/2^{e_m}$ of $B^{\pm}(x_m)$ modulo a
power of 2? It is not hard to see that this is strictly related to knowledge of the sequence $s$.
\end{enumerate}

\noindent
\textbf{Acknowledgements}
We wish to thank the referees for several useful suggestions, Victor Adamchick (who first introduced the first author
to this problem), and the
Louisiana Optical Network Institute (LONI) for partial support. Special thanks to Victor Moll and
Tewodros Amdeberhan for help and discussions spanning many years.

\end{document}